\documentclass[11pt]{amsart}

\usepackage{amsmath}
\usepackage{amsthm}
\usepackage{amsbsy}
\usepackage{amssymb}
\usepackage{mathtools}
\usepackage{calrsfs}
\usepackage{comment}
\usepackage[shortlabels]{enumitem}
\usepackage[normalem]{ulem}
\usepackage{amsrefs}
\usepackage{tikz}
\usetikzlibrary{positioning,patterns,calc}
\usepackage{mathrsfs}
\DeclareMathAlphabet{\mathpzc}{OT1}{pzc}{m}{it}
\usepackage[utf8]{inputenc}
\usepackage{hyperref}

\newcommand{\vol}{\textnormal{vol}}

\def\R{\mathbb{R}}

\newtheorem{theorem}{Theorem}[section]
\newtheorem{lemma}[theorem]{Lemma}
\newtheorem{proposition}[theorem]{Proposition}
\newtheorem{corollary}[theorem]{Corollary}

\newtheorem{definition}[theorem]{Definition}
\newtheorem*{theorem*}{Theorem}
\newtheorem*{remark}{Remark}

\makeatletter
\@addtoreset{claim}{theorem}
\makeatother

\theoremstyle{plain}

\theoremstyle{plain}

\usepackage{enumitem}
\usepackage{lipsum}
\setlist[itemize]{leftmargin=*}

\def\bpf{\begin{proof}}
	\def\epf{\end{proof}}
\def\be{\begin{equation}}
	\def\ee{\end{equation}}
\def\bea{\begin{eqnarray}}
	\def\eea{\end{eqnarray}}
\def\bt{\begin{theorem}}
	\def\et{\end{theorem}}
\def\bl{\begin{lemma}}
	\def\el{\end{lemma}}
\def\br{\begin{remark}}
	\def\er{\end{remark}}
\def\bc{\begin{corollary}}
	\def\ec{\end{corollary}}
\def\bd{\begin{definition}}
	\def\ed{\end{definition}}
\def\bp{\begin{proposition}}
	\def\ep{\end{proposition}}

\def\Prob{\textnormal{Prob}^g_{\text{WP}}}
\def\Exp{\textnormal{E}^g_{\text{WP}}}
\def\Probb{\textnormal{Prob}^{g,n}_{\text{WP}}}
\def\Expp{\textnormal{E}^{g,n}_{\text{WP}}}

\newcommand{\eg}{\textit{e.g.\@ }}

\title{Large Steklov Eigenvalues on Hyperbolic surfaces}
\author{Xiaolong Hans Han}
\address{Yau Mathematics Science Centre, Tsinghua University, Beijing, China, 100084}
\email{xlhan@mail.tsinghua.edu.cn}

\author{Yuxin He}
\address{Yau Mathematics Science Centre, Tsinghua University, Beijing, China, 100084}
\email{hyx21@mails.tsinghua.edu.cn}

\author{Han Hong}
\address{School of Mathematics and Statistics, Beijing Jiaotong University, Beijing, China, 100091}
\email{honghan0927@gmail.com}

\date{September 2023}

\begin{document}
	
	\maketitle
	\begin{abstract}
		In this paper, we first construct a sequence of hyperbolic surfaces with connected geodesic boundary such that the first normalized Steklov eigenvalue $\tilde{\sigma}_1$ tends to infinity. We then prove that as $g\rightarrow \infty$, a generic $\Sigma\in \mathcal{M}_{g,n}(L_g)$ satisfies  $\tilde{\sigma}_1(\Sigma)>C\cdot \|L_g\|_1$ where $C$ is a positive universal constant. Here  $\mathcal{M}_{g,n}(L_g)$ is the moduli space of hyperbolic surfaces of genus $g$ and $n$ boundary components of length $L_g=(L_g^1,\cdots, L_g^n)$  endowed with the Weil-Petersson metric where $\|L_g\|_1\rightarrow\infty$ satisfies certain conditions.
	\end{abstract}
	
	\section{Introduction}
	Let $\Sigma$ be an orientable, compact, connected Riemannian surface with boundary. The Steklov eigenvalue problem is given by 
	\[\begin{cases}
		\Delta u=0 \ \ \ \ \ \text{on}\ \Sigma\\
		\partial_\eta u=\sigma u\ \ \ \text{on} \ \partial\Sigma
	\end{cases}\]
	where $\Delta$ is the Laplace–Beltrami operator on $\Sigma$ and $\eta$ is the outer unit normal vector along the boundary of $\Sigma$. The eigenvalues are ordered with multiplicity as
	\[0=\sigma_0<\sigma_1(\Sigma)\leq\sigma_2(\Sigma)\leq\ldots\rightarrow \infty.\]
	
	Consider the scaling invariant quantity, the first (nonzero) normalized Steklov eigenvalue, that is defined by $\tilde{\sigma}_1(\Sigma)=\sigma_1(\Sigma)\ell(\partial\Sigma)$ where $\ell$ denotes the length. Let $S(g,k)$ denote the class of the compact surfaces of genus $g$ with $k$ boundary components. The optimization of $\tilde{\sigma}_1(\Sigma)$ in $S(g,k)$ has been intensively studied during the last two decades. The following upper bound was given by Fraser and Schoen \cite{FS2012} 
	\begin{equation}\label{upperbound}\tilde{\sigma}_1(\Sigma)\leq 2\pi(g+k).\end{equation}
	For simply connected planar
	domains, inequality (\ref{upperbound}) is sharp and was proved by Weinstock in \cite{W1954}. On the other hand, an upper bound independent of the number of boundary components which makes it an estimate sharper than (\ref{upperbound}) for $k$ large enough was obtained by Kokarev in \cite{Kokarev}:
	\begin{equation}\label{upperbound2}
		\tilde{\sigma}_1\leq 8\pi(g+1).
	\end{equation}
	
	In this paper, we are concerned with compact surfaces with arbitrarily large first normalized Steklov eigenvalue $\tilde{\sigma}_1.$ A. Girouard and I. Polterovich raised a question in \cite{GP2012}: \textit{Is there a sequence $\Sigma_N$ of surfaces with boundary of genus $N$ tending to infinity such that $\tilde{\sigma}_1(\Sigma_N)\rightarrow \infty$  as $N \rightarrow\infty$? If yes, is it possible to achieve linear growth?} According to (\ref{upperbound}) and (\ref{upperbound2}), large $\tilde{\sigma}_1$ implies large genus. B. Colbois and A. Girouard proved in \cite{CG2018} that
	there exist a sequence  of compact surfaces $\Sigma_N$ with boundaries and a constant $C>0$ such that for each $N\in\mathbb{N}$, the genus of $\Sigma_N$ is $N+1$ and 
	\[\tilde{\sigma}_1(\Sigma_N)\geq C\cdot N.\]
	This result implies that $(\ref{upperbound2})$ is optimal. Their proof involves a comparison of eigenvalues on graphs and Steklov eigenvalues. Not only the genus but also the number of boundary components of the surface $\Sigma_N$ tend to infinity as $N\rightarrow \infty.$ It is interesting to know if the growth of the boundary components is necessary. Thus a natural question raised in \cite{CG2018} is whether there exists a sequence of compact surfaces with a fixed number of boundary components such that the first normalized Steklov eigenvalue becomes arbitrarily large. Colbois-Girouard-Raveendran \cite{CAR2018} constructed a sequence of surfaces with connected boundary such that its first normalized Steklov eigenvalue grows linearly to infinity.
	
	Based on the work from Mirzakhani \cites{Mirz13} and Nie-Wu-Xue \cite{NWX20} and inspired by the work  \cite{hewu2022} in which the second author with Wu constructs sequences of  Riemann surfaces with geodesic boundary such that the  first Laplacian eigenvalues of Neumann type have a uniform gap from 0, we prove the following theorem.
	\begin{theorem}\label{maintheorem1}
		There exist a sequence  of compact hyperbolic surfaces $\Sigma_N$with one geodesic boundary component and a uniform constant $C>0$ such that the genus of $\Sigma_N$ is $N-1$ and 
		\[\tilde{\sigma}_1(\Sigma_N)\geq C\cdot\log (N+1).\]
	\end{theorem}
	
	Notice that the sequence of surfaces has intrinsic curvature $-1$ and connected geodesic boundary. On the one hand, although the construction from \cite{CG2018,CAR2018} does not yield hyperbolic surfaces, modifications by quasi-isometries can yield hyperbolic surfaces. Thus Theorem \ref{maintheorem1} provides a different existence proof  from a probabilistic way. On the other hand, \cite{CG2018,CAR2018} give a sequence of surfaces with a linear genus growth. In this sense, our lower bound is not sharp. Nevertheless, it provides a different and new point of view to construct surfaces with large Steklov eigenvalues. Moreover, it motivates our second result about random surfaces. The proof of Theorem \ref{maintheorem1} relying on a modified Jammes's inequality and the estimates of two constants will be very important in the proof of Theorem \ref{maintheorem2} below.
	
	Regarding Theorem \ref{maintheorem1}, we can also prescribe the number $l$ of boundary components. Indeed, by removing $\ell-1$ small discs centered at different interior points $p_i\in \Sigma_N$ for $i=1,\cdots,\ell-1$ and modifying the metric slightly around those discs we can make the first normalized Steklov eigenvalue strictly larger than the one of $\Sigma_N$. However, the curvatures of surfaces in the sequence are no longer constant $-1$ and the new $\ell-1$ boundary components are not geodesic anymore. Readers can refer Proposition 4.3 in \cite{FS2016} for detailed arguments. 
	
	We now state our second result. Not only are we able to show the existence of a sequence of compact hyperbolic surfaces with connected geodesic boundary such that the first normalized Steklov eigenvalue tends to infinity as stated in Theorem \ref{maintheorem1}, but we also show that the probability that a random Riemann surface (with fixed finite number of boundary components) has a large first normalized Steklov eigenvalue is asymptotically one. For $L=(L^1,\cdots,L^n)\in \mathbb{R}_{\geq 0}^n$, let $\mathcal{M}_{g,n}(L)$ be the moduli space of hyperbolic surfaces of genus $g$ and with $n$ geodesic boundaries of length $L$. Let $\Probb$ be the probability measure with respect to the Weil-Petersson metric whose precise definition will be given in section \ref{weildefinition}. We view the first normalized Steklov eigenvalue as a nonnegative random variable on $\mathcal{M}_{g,n}(L)$, and have the following theorem.
	\begin{theorem}\label{maintheorem2}
		Assume that $\{L_g\}_{g\geq 2}$ is any  sequence of $n$-tuples with $L_g=(L_g^1,\cdots,L_g^n)$ such that $$\lim\limits_{g\to \infty}\sum_{i=1}^n L_g^i=\infty,  \quad L_g^i>1\quad \text{and }\quad\limsup_{g\rightarrow \infty}\frac{\sum_{i=1}^nL_g^i}{\log g}<1.$$ 
		Then there exists a universal constant $C$ such that 
		\[\lim \limits_{g\to \infty}\Probb\Big(X_{g,n}\in \mathcal{M}_{g,n}(L_g);\ \ \tilde{\sigma}_1(X_{g,n})\geq C\cdot \sum_{i=1}^nL_g^i  )=1.\]
	\end{theorem}

A direct consequence of Theorem \ref{maintheorem2} is that for any fixed integer $n\geq 1$, there exist infinitely  many sequences of hyperbolic surfaces with $n$ boundary components such that their first normalize Steklov eigenvalues grow to infinity as the genus tends to infinity.
	
	We briefly explain the proof of Theorem \ref{maintheorem2}. It relies on the proof of Theorem \ref{maintheorem1} and two other results: Theorem \ref{randomHj} and Theorem \ref{randomH} in which we prove that the probability that a random hyperbolic surface has the Cheeger's constant and the modified Jammes's constant uniformly bounded away from zero is asymptotically one, respectively. Most of the proof concentrates on computing the probability of two constants in Section 5.1 and Section 5.2. 
	
	However, since we are dealing with hyperbolic surfaces of (possibly) more than one boundary component, there are some differences.  As observed in Part I of the  proof of Theorem \ref{maintheorem1}, if the surface only has one boundary component, the large width of the half-collar of the boundary component implies positive lower bound of the modified Jammes's constant. This conclusion is not necessarily true in multiple boundary cases since the components of interior boundary $\partial_I\Omega$ (see definition in Section 2.1) of the admissible subset $\Omega$ could consist of arcs contained in half-collar of boundary and closed geodesics in the interior of the surface such that the total length of interior boundary $\partial_I\Omega$ of $\Omega$ is ``small''. Then the analysis conducted in case 1 and case 2 of Part I proof of Theorem \ref{maintheorem1} fails. However, in this scenario, the union of closed nontrivial  curves of $\partial_I\Omega$  gives a separating multicurve that separates the surface into two components. In Proposition 5.4 we show such surface is not ``common" in moduli space as genus goes to infinity. Hence we are left with the case in which the total length of $\partial_I\Omega$ is ``large''.  In this case, a positive lower bound of modified Jammes's constant follows from the definition.

	\subsection*{Notations.} 
	For two functions $f_1$ and  $f_2$, if there exists a uniform constant $C>0$ independent of $g$ with the relationship $f_1(g)\leq C\cdot f_2(g)$ holding  for any $g$, then we say  
	$$f_1(g)\prec f_2(g) \quad \emph{or} \quad f_2(g)\succ f_1(g).$$
	We say $f_1\asymp f_2$ if $f_1\prec f_2$ and $f_1\succ f_2.$
	
	\subsection*{Acknowledgments}
	The authors would like to thank Professor Yunhui Wu for introducing this problem and helpful discussion. They also want to express gratitude to Yuhao Xue for discussion. We thank anonymous referees for their helpful comments and suggestions. This work is supported by Shuimu Tsinghua Scholar Program, China Postdoctoral Science Foundation No. 2021TQ0186
	and International Postdoctoral Exchange Fellowship Program No. YJ20210267.
	
	\section{Preliminaries}
	\subsection{Cheeger's constant and a modified Jammes's constant}
	Let $(\Sigma,\partial\Sigma)$ be a compact Riemannian surface with boundary. The classical Cheeger's constant (of Neumann type) is defined by
	\[h_C(\Sigma)=\inf_{\Omega}\max \left\{\frac{\ell(\partial_I\Omega)}{|\Omega|},\frac{\ell(\partial_I\Omega)}{|\Sigma|-|\Omega|}\right\},\]
	where $\Omega$ is a compact subset of $\Sigma$ and $\partial_I\Omega=\partial\Omega\cap \text{Int}(\Sigma).$ Here $\ell(\cdot)$ and $|\cdot|$ denote the length and the area with respect to the corresponding metric, respectively. It is well known (e.g. Lemma 8.3.6 in \cite{buser2010geometry}) that by Yau's observation we may restrict $\Omega$ to the set family for which $\Omega$ and $\Sigma\setminus \Omega$ are connected in the definition of $h_C$. To study the lower bound of Steklov eigenvalue, P. Jammes \cite{Jammesinequality} introduced the following Cheeger's type constant
	\[h_J(\Sigma)=\inf_{|\Omega|\leq |\Sigma|/2} \frac{\ell(\partial_I\Omega)}{\ell(\partial_E\Omega)},\]
	which is slightly different from Escobar's constant in \cite{Escobar}. Here $\Omega$ runs over compact subsets of $\Sigma$ and $\partial_E\Omega=\Omega\cap \partial \Sigma\neq \emptyset.$  For our purpose, we can slightly change the definition by adding restrictions on $\Omega$: $\partial_E\Omega\neq \emptyset$ and all components of $\Omega^c$ intersect $\partial\Sigma.$ We denote it by $\tilde{h}_J(\Sigma)$ for distinction and call it the modified Jammes's constant.
	
	We have the following Lemma. The proof is exactly the same as one of P. Jammes.
	\begin{lemma}[\cite{Jammesinequality}]\label{Jammes}
		The first nonzero Steklov eigenvalue of $\Sigma$ is bounded below in terms of the Cheeger's constant $h_C(\Sigma)$ and the modified Jammes's constant $\tilde{h}_J(\Sigma)$: 
		\[\sigma_1(\Sigma)\geq \frac{1}{4}h_C(\Sigma)\cdot \tilde{h}_J(\Sigma).\]
	\end{lemma}
	Although this lemma holds for higher dimensional manifolds with boundary, we only need it for surfaces. For the convenience of readers, we include the proof below.
	\begin{proof}
		Let $f$ be the first eigenfunction corresponding to the eigenvalue $\sigma_1$. Denote $\Sigma_+=f^{-1}([0,\infty))$ such that $|\Sigma_+|\leq \frac{1}{2}|\Sigma|$; otherwise replace $f$ with $-f$.  Then
		\begin{align*}
			\sigma_1(\Sigma)&=\frac{\int_{\Sigma_+}|\nabla f|^2}{\int_{\partial\Sigma_+}f^2}\\
			&\geq \frac{\left(\int_{\Sigma_+}f|\nabla f|\right)^2}{\int_{\Sigma_+}f^2\int_{\partial\Sigma_+}f^2}\\
			&=\frac{1}{4}\frac{\int_{\Sigma_+}|\nabla f^2|}{\int_{\Sigma_+}f^2}\frac{\int_{\Sigma_+}|\nabla f^2|}{\int_{\partial\Sigma_+}f^2}.
		\end{align*}
		Therefore using the Co-area formula, we have
		\[\int_{\Sigma_+}|\nabla f^2|=\int_0^\infty\ell(\partial_I D_t)dt\geq h_C\int_0^\infty |D_t|dt=h_C\int_{\Sigma_+}f^2\]
		where $D_t=f^{-1}([\sqrt{t},\infty))$. Since $D_t$ intersects $\partial\Sigma$ and all components of $D_t$ intersects $\partial\Sigma$ by the maximum principle, we have
		\[\int_0^\infty\ell(\partial_I D_t)dt\geq \tilde{h}_J\int_0^\infty \ell(\partial_E D_t)dt=\tilde{h}_J\int_{\partial\Sigma_+}f^2.\]
		The proof is complete.
	\end{proof}

Perrin considered a somewhat similar Jammes's type constant together with an inequality in \cite{perrin}. In both papers, what matters is an inequality that relates the first Steklov eigenvalue to the Cheeger's constant and a constant which reflects the geodesic information of hyperbolic surfaces with boundary.
	
	\subsection{The Weil-Petersson metric, probability measure.}\label{weildefinition}
	
	Let $S_{g,n}$ be a closed hyperbolic surface of genus $g\geq 2$ and with $n$ boundary components of lengths $L=(L_1,\cdots,L_n)\in \mathbb{R}_{\geq 0}^n$. Let $\mathcal{T}_{g,n}(L)$ be  the Teichm\"uller space of closed hyperbolic surfaces of genus $g$ with $n$ boundary components of length $L$.
 Here $L_i=0$ represents a cusp instead of a geodesic boundary component.
	The moduli space  $\mathcal{M}_{g,n}(L)$ is defined as $\mathcal{T}_{g,n}(L)/\text{Mod}_{g,n}$ where $\text{Mod}_{g,n}$ is the mapping class group of $S_{g,n}$ fixing the order of  boundary components. We write $\mathcal{M}_g=\mathcal{M}_{g,0}$ for simplicity. 
	
	The Teichm\"uller space is endowed with the Weil-Petersson metric and the Weil-Petersson symplectic form has a natural form in Fenchel-Nielsen coordinates \cite{Wolpert} proved by Wolpert:
	\[w_{\text{WP}}=\sum_{i=1}^{3g+n-3}d\ell_{\alpha_i}\wedge d\tau_{\alpha_i},\]
	where $(\ell_{\alpha_i},\tau_{\alpha_i})_{i=1}^{3g+n-3}$ are the global coordinates for the Teichm\"uller space $\mathcal{T}_{g,n}(L).$ The corresponding Weil-Petersson volume form is written as \[d\vol_{\text{WP}}=\frac{1}{(3g+n-3)!}\wedge^{3g+n-3}w_{\text{WP}}.\]
	This measure is invariant under the mapping class group, hence induces a measure on the moduli space $\mathcal{M}_{g,n}$, still denoted by $d\vol_{\text{WP}}.$ According to \cite{Mirz13}, the probability measure $\text{Prob}^{g,n}_{WP}$ is defined by
	\[\text{Prob}^{g,n}_{\text{WP}}(\Omega)=\frac{1}{\vol(\mathcal{M}_{g,n}(L))}\int_{\mathcal{M}_{g,n}}\chi_\Omega \ d\vol_{\text{WP}}.\]
	where $\Omega\subset \mathcal{M}_{g,n}(L)$ is a Borel subset, $\chi_\Omega$ is the characteristic function on $\Omega$ and $\vol(\mathcal{M}_{g,n}(L))$ is the total volume of moduli space which is finite.  For a measurable function $f$ on $\mathcal{M}_{g,n}(L)$, the expected value of $f$ is defined by \[
\Expp[f]=\frac{1}{\vol(\mathcal{M}_{g,n}(L))}\int_{\mathcal{M}_{g,n}}f \ d\vol_{\text{WP}}.
\]
Let $\text{Prob}^{g}_{\text{WP}}(\Omega)=\text{Prob}^{g,0}_{\text{WP}}(\Omega)$ and $\Exp[f]=\text{E}^{g,0}_{\text{WP}}[f]$.  Readers can refer to \cite{Mirz10} for more details.
	
	We first recall some results from Mirzakhani with her coauthors and Nie-Wu-Xue. Denote $V_{g,n}(x_1,\cdots,x_n)$ the Weil-Petersson volume of the moduli space $\mathcal{M}_{g,n}(x_1,\cdots,x_n)$ and $V_{g,n}=V_{g,n}(0,\cdots,0)$. Mirzakhani showed that $V_{g,n}(x_1,\cdots,x_n)$ is a polynomial in $x_1^2,\cdots, x_n^2$ with degree $3g+n-3$. In particular, as defined in \cite{Mirz07} $V_{0,3}(x,y,z)=1$.
	
	\begin{lemma}\label{Mirz13lemma}
		(1)\cite[Lemma 3.2]{Mirz13}
		\[V_{g,n}\leq V_{g,n}(x_1,\cdots,x_n)\leq e^{\frac{x_1+\cdots+x_n}{2}}V_{g,n}.\]
		(2)\cite[Lemma 3.2]{Mirz13} For $g,n\geq 0$,
		\[V_{g,n+4}\leq V_{g+1,n+2}.\]
		(3)\cite[Theorem 3.5]{Mirz13} For fixed $n\geq 0$, as $g\rightarrow \infty$ we have
		\[\frac{V_{g,n+1}}{2gV_{g,n}}=4\pi^2+O\left(\frac{1}{g}\right),\]
		\[\frac{V_{g,n}}{V_{g-1,n+2}}=1+O\left(\frac{1}{g}\right),\]
		where the implied constants are related to $n$ and independent of $g$.
	\end{lemma}
	
	The following lemma improves $(1)$ of Lemma $\ref{Mirz13lemma}$.
	\begin{lemma}\cite[Lemma 22]{NWX20}\label{NWXlemma} Let $n\geq 1$, then
there exists a constant $c=c(n)>0$ such that for any $g\geq 1$ and 
		 $x_1,\cdots,x_n\geq 0$,  we have 
		\[\prod_{i=1}^n\frac{\sinh(x_i/2)}{x_i/2}(1-c(n)\frac{\sum x_i^2}{g})\leq\frac{V_{g,n}(x_1,\cdots,x_n)}{V_{g,n}}\leq \prod_{i=1}^n\frac{\sinh(x_i/2)}{x_i/2}.\]
	\end{lemma}
	We also need the following lemma in the main proof.
	\begin{lemma}\cite[Corollary 3.7]{Mirz13} \label{mir13corollary3.7}
		Let $b,k\geq 0$ and $C<C_0=2\ln 2$, then
		\[\sum_{\substack{g_1+g_2=g+1-k\\r+1\leq g_1\leq g_2}}e^{Cg_1}\cdot g_1^b\cdot V_{g_1,k}\cdot V_{g_2,k}\asymp \frac{V_g}{g^{2r+k}}\]
		as $g\rightarrow \infty.$
	\end{lemma}
	
	\section{Proof of Theorem \ref{maintheorem1}}
	
	We first introduce some results that are essential to our proof. Let $\omega: \{2,3,\cdots\}\to \R^{>0}$ be a function satisfying
	\begin{equation}\label{cond omega}
		\lim_{g\rightarrow\infty} \omega(g) = +\infty \ \text{ and} \ \lim_{g\rightarrow\infty}\frac{\omega(g)}{\log\log(g)} =0.
	\end{equation}
	\begin{theorem}[{\cite[Theorem 1 and 2]{NWX20}}]\label{NXW}
		Let $\omega(g)$ be a function satisfying \eqref{cond omega} and $\epsilon>0$. Consider the following three conditions for $X \in\mathcal{M}_g$:
		\begin{enumerate}
			\item $|\ell_{\rm sys}^{\rm sep}(X) - (2\log g-4\log\log g)| \leq \omega(g)$;
			\item $\ell_{\rm sys}^{\rm sep}(X)$ is achieved by a simple closed geodesic $\gamma$ that separates the surface into $S_{1,1}\cup S_{g-1,1}$;
			\item There is a half-collar for $\gamma$ in the $S_{g-1,1}$-part of $X$ with width $\frac{1}{2}\log g - (\frac{3}{2}+\epsilon)\log\log g$.
		\end{enumerate}
		Then we have
		$$\lim_{g\rightarrow\infty} \Prob\left(X\in\mathcal{M}_g\ ;\ X \text{ satisfies } (1),\ (2) \text{ and } (3)\right) =1.$$
	\end{theorem}
	
	\begin{remark}
		A direct corollary of Theorem \ref{NXW} is that  the expectation of the unbounded random 
		variable $\ell_{\rm sys}^{\rm sep}(X)$ grows at least of order $2\log g$ as $g$ tends to 
		infinity. In \cite{parlier2021simple} it was proved that the expectation is asymptotically $2\log g$ exactly.
	\end{remark}
	
	\begin{theorem}[{\cite[Theorem 4.8]{Mirz13}}]\label{Mirz13}
		Let $C<\frac{\log 2}{2\pi+\log 2}$, then
		\begin{equation*}
			\lim_{g\rightarrow\infty}\Prob\left(X\in\mathcal{M}_g\ ;\ h_C(X)<C\right)=0.
		\end{equation*}
	\end{theorem}
	
	\begin{remark}
		Note that $h_C$ in Theorem \ref{Mirz13} is Cheeger's constant for closed surface $X$. It is the same definition as one in Section 2.1 only by allowing $\partial\Sigma=\emptyset.$
	\end{remark}
	We now start the proof.
	
	\begin{proof}[Proof of Theorem \ref{maintheorem1}]
			
		According to Theorem \ref{NXW} and Theorem \ref{Mirz13}, for any  positive constant $\epsilon,$  for large enough fixed genus $g$ there exists a closed hyperbolic surface $X_g$ and a closed geodesic $\gamma$ on $ X_g$ realizing the separating systole of it and  separating $X_g$ into a union of $S_{1,1}$ and $S_{g-1,1}$, i.e., a compact surface with one genus and one boundary component and a compact surface with $g-1$ genus and one boundary component. Moreover, the Cheeger's constant of $X_g$ satisfies
		\begin{equation}\label{lowerboundforcheegerconstant}
			h_C(X_g)>\frac{\log 2}{2\pi+\log 2}-\epsilon.
		\end{equation}
		We will show that the compact surface $S_{g-1,1}$ is the desired surface $\Sigma$ with genus going to infinity such that 
		\[\sigma_1(\Sigma)\cdot \ell(\partial \Sigma)\rightarrow \infty,\ \ \ \text{as}\ \ \ g\rightarrow \infty.\]
		
		Hereafter, denote $S_{g-1,1}$ by $\Sigma$. By Theorem \ref{NXW}, $\partial\Sigma$ is a single closed geodesic and $(2-\epsilon)\log g\leq \ell(\partial \Sigma)\leq 2\log g$. Moreover, $\gamma=\partial\Sigma$ has a half-collar with width at least $(\frac{1}{2}-\epsilon)\log g$ inside $\Sigma$. Hence, all we need to prove is that $\sigma_1(\partial\Sigma)$ is bounded below uniformly by a positive constant. Due to Lemma \ref{Jammes}, it suffices to show that $h_C$ and $\tilde{h}_J$ of $\Sigma$ are uniformly bounded from below.
		
		\textbf{Part I:} We first deal with $\tilde{h}_J(\Sigma)$. We denote the half-collar by $C(\gamma)$ for simplicity.
		
	\begin{itemize}
		\item Case 1: The interior boundary $\partial_I \Omega$ is contained in  $C(\gamma).$ The exterior boundary $\partial_E\Omega$ separates $\gamma$ into $\alpha$ and $\beta$, i.e., $\gamma=\alpha\cup\beta$ and $\partial_E\Omega=\alpha$. It is readily seen that $\Omega$ is also contained in $C(\gamma)$. If not, the area would exceed half of the area of $\Sigma$ by simple calculations, which is excluded by the definition of $h_J$. This fact ensures that $\beta\cup\partial_I\Omega$ is not contractible. Since $\gamma$ is a geodesic, then $\ell(\partial_I\Omega)\geq \ell(\alpha)$, otherwise $\beta\cup\partial_I\Omega$ will bound a domain containing $S_{1,1}$ and contained in $S_{1,1}\cup C(\gamma)$. Replacing each component of it with the unique geodesic in the homotopic class, we find a geodesic bounding $S_{1,1}$ shorter than $\gamma$,
		contradicting the minimizing property of $\gamma$. That is, 
		\[\frac{\ell(\partial_I\Omega)}{\ell(\alpha)}\geq 1.\]
		\item Case 2: The interior boundary passes through the complement of the half-collar, i.e., $\partial_I\Omega\cap (\Sigma\setminus C(\gamma))\neq \emptyset$. In this case, $\ell(\partial_I \Omega)\geq (1-2\epsilon)\log g$. Apparently $\ell(\alpha)\leq 2\log g$, thus
		\[\frac{\ell(\partial_I\Omega)}{\ell(\alpha)}\geq (1-2\epsilon)/2.\]
	\end{itemize}
		According to the definition of $\tilde{h}_J(\Sigma)$, only the above two  cases need to be considered. Then in either case, we have that $\tilde{h}_J(\Sigma)\geq C_1>0.$
		
		\textbf{Part II:} We now estimate $h_C(\Sigma)$ from below which is more difficult than the estimate of $h_J(\Sigma)$. To do so, we need to compare Cheeger's constants of $\Sigma$ and $X_g$. Consider the class of all compact subsets $\Omega$ of $\Sigma$ and divide it into two subclasses A and B:
		\[A=\{\Omega\subset \Sigma: \partial\Omega\cap \gamma=\emptyset\}\ \text{and}\ B=\{\Omega\subset \Sigma: \partial\Omega\cap \gamma\neq \emptyset\}.\]
		
		\vskip.3cm
		\begin{itemize}
			\item Case 1: $|\Omega|\leq|\Omega^c|$.  By definition, we have
			\begin{equation}\label{A1}
				\inf_{\Omega\in A,\ |\Omega|<|\Omega^c|}\frac{\ell(\partial \Omega)}{|\Omega|}\geq h_C(X_g).
			\end{equation}
			When $\Omega\in B$, we have two cases. If $\Omega\cap \gamma\neq \emptyset$, denote the intersection by $\alpha$ and its complement by $\beta=\gamma\setminus \alpha$. By the definition again,
			\[h_C(X_g)\leq \frac{\ell(\alpha)+\ell(\partial_I\Omega)}{|\Omega|}.\]
			On the other hand, if $\partial_I\Omega$ is completely contained in the half-collar of $\gamma$, we also have  that $\ell(\alpha)\leq \ell(\partial_I\Omega)$. Thus $$h_C(X_g)\leq \frac{2\ell(\partial_I\Omega)}{|\Omega|}.$$ Otherwise, since $\partial_I\Omega$ passes through the half-collar at least twice, $\ell(\partial_I\Omega)\geq (1-2\epsilon)\log g$. As $\ell(\alpha)\leq2\log g$, it yields that $$\ell(\alpha)\leq \frac{2}{1-2\epsilon}\ell(\partial_I\Omega).$$
			Hence $$h_C(X_g)\leq \frac{3-2\epsilon}{1-2\epsilon}\ell(\partial_I\Omega)/|\Omega|.$$
			In both cases, from Theorem \ref{NXW} it follows that
			\begin{equation}\label{A2}
				\inf_{\Omega\in B,\ |\Omega|<|\Omega^c|}\ell(\partial_I\Omega)/|\Omega|\geq C\cdot h_C(X_g)
			\end{equation}
			for a positive universal constant $C$ depending only on $\epsilon$.
			
			\vskip.4cm
			\item Case 2: $|\Omega|>|\Omega^c|$.
			
			When $\partial\Omega\cap \gamma=\emptyset$, we get
			\[h_C(X_g)\leq \max\{\frac{\ell(\partial\Omega)}{|\Omega^c|+2\pi},\frac{\ell(\partial\Omega)}{|\Omega|}\}\leq \frac{\ell(\partial \Omega)}{|\Omega^c|}\]
			where the first inequality holds since by the Gauss-Bonnet equation we have that
			$|S_{1,1}|=2\pi$ and the second inequality follows by a simple comparison. Thus
			\begin{equation}\label{B1}
				\inf_{\Omega\in A,|\Omega|>|\Omega^c|}\frac{\ell(\partial \Omega)}{|\Omega^c|}\geq h_C(X_g).
			\end{equation}
			When $\partial\Omega\cap \gamma\neq \emptyset$, $\alpha$ and $\beta$ are defined same as before. Note that $\beta\cup \partial_I\Omega$ bounds the domain $\Omega^c$. Then
			\[h_C(X_g)\leq \frac{\ell(\beta)+\ell(\partial_I\Omega)}{|\Omega^c|}.\]
			It is not difficult to follow the analysis in Case 1 depending on whether $\partial_I\Omega$ is contained in the half-collar or not to obtain  that
			\begin{equation}\label{B2}
				\inf_{\Omega\in B,\ |\Omega|>|\Omega^c|} \frac{\ell(\partial_I\Omega)}{|\Omega^c|}\geq C\cdot h_C(X_g)
			\end{equation}
			where $C$ is a positive constant depending solely on $\epsilon.$
		\end{itemize}
		\vskip.3cm
		Finally, combining equations (\ref{A1}),(\ref{A2}),(\ref{B1}) and (\ref{B2}), we can get that
		\[h_C(\Sigma)=\inf_{\Omega\in A\cup B}\max \{\frac{\ell(\partial_I\Omega)}{|\Omega|},\frac{\ell(\partial_I\Omega)}{|\Omega^c|}\}\geq \min(C,1)\cdot h_C(X_g)\]
		Hence it follows from (\ref{lowerboundforcheegerconstant}) that
		\[h_C(\Sigma)\geq \min\{C,1\}\cdot \left(\frac{\log2}{2\pi+\log 2}-\epsilon\right)>0.\]
		The proof is complete.
	\end{proof}
	
	\section{Isoperimetric curves}
	
	The following theorem gives a characterization of minimizers of the isoperimetric problem in compact hyperbolic surfaces with geodesic boundary, generalizing the result in \cite{Morgan}.  
	
	\begin{theorem}\label{Isoperimetrisolution}
		Let $S$ be a compact hyperbolic surface with geodesic boundaries. For a given value $0<A<|S|$, a perimeter-minimizing system of embedded rectifiable curves bounding a region $\Omega$ of area $A$ consists of a set of curves of one of the following types.
		
		If the region has an empty intersection with the boundary of $S$,
		\begin{itemize}[leftmargin=1cm]
			\item[(1)] a circle,
			\item[(2)] two neighboring curves at a constant distance from a simple closed geodesic (equidistant curves), bounding an annulus or complement,
			\item[(3)] closed geodesics or single neighboring curves.
		\end{itemize}
		If the region has a nonempty intersection with the boundary of $S$,
		\begin{itemize}[leftmargin=1cm]
			\item[(4)] a half-circle with free boundary,
			\item[(5)] two free boundary neighboring arcs at a constant distance from a simple free boundary geodesic arc, bounding a half annulus or complement,
			\item[(6)] free boundary geodesic arcs and closed geodesics,
			\item[(7)] single free boundary equidistant arcs and single equidistant curves from simple closed geodesics.
		\end{itemize}
		Note that all curves in the set have the same curvature.
	\end{theorem}
	\begin{proof}
		The classical existence and regularity theorems from \cite{Almgren,Gruter,Tamanini} tell that there exists a perimeter minimizer among regions of the prescribed area bounded by embedded smooth curves with same constant curvature and, if these curves have a nonempty intersections with the boundary $S$, then they meet orthogonally along intersections. The possibility of cases (1),(2) or (3) have been shown in \cite{Morgan}. There are two important facts. One is that the region $\Omega$ has to be an annulus if it has an annulus component. The other one is that neighboring curves are not equidistant curves from nonsimple closed geodesic since they are not smooth, contradicting the regularity part of the theorem from \cite{Almgren,Gruter,Tamanini}. 
		
		We now explain cases (4),(5),(6),(7). 
		Lifting the surface to the universal covering, we know that a curve with constant curvature smaller than 1 is one of the equidistant lines around a geodesic and a curve with constant curvature higher than 1 is a circle. Firstly, by doubling the surface, we know that we can not have more than one half-circle with free boundary according to (1). Moreover, the half annulus or the complement is the only component of the region $\Omega$ if it contains half annulus or the complement according to (2). Lastly, whenever there are intersections between the curves and the boundary of the surface $S$, they meet perpendicularly. Thus the proof is complete.
	\end{proof}
	
	For compact hyperbolic surface $\Sigma$ with geodesic boundaries, we define the geodesic Cheeger's constant  as $$
	H(\Sigma)=\inf\limits_{\Gamma}\max\left\{\frac{\ell(\Gamma)}{|A|},\frac{\ell(\Gamma)}{|B|}\right\},  
	$$
	where $\Gamma$ consists of disjoint simple  free boundary geodesic arcs and simple closed geodesics that separate $\Sigma$ into two connected components $A$ and $B$. We call a domain such as $A$ or $B$ bounded by such $\Gamma$  a half surface.
	
	\begin{lemma}\label{hgeqHH+1}
For $L=(L_1,\cdots,L_n)\in \mathbb{R}_{\geq 0}^n$,
		let $\Sigma\in \mathcal{M}_{g,n}(L)$ be a hyperbolic surface of genus $g$ and $n$ geodesic boundary components with lengths $L$. Then 
		$$
		h_C(\Sigma)\geq \frac{H(\Sigma)}{H(\Sigma)+1}.
		$$
	\end{lemma}
	\begin{proof}
		We assume that $\Gamma$ is the set of constant curvature curves which achieves the Cheeger's constant $h_C$, so $\Gamma$ must separate the surface into two connected components. According to Theorem \ref{Isoperimetrisolution}, $\Gamma$ is one of  cases (1)-(7). This inequality has been shown by Mirzakhani \cite{Mirz13} for cases (1)-(3). For other cases, the idea is similar so we briefly include it here. For example, we readily have $h_C\geq 1$ in cases (4)-(5) by applying the isoperimetric inequality to the double of the domain along the geodesic arc.
		The inequality is obvious in case (6) since $h_C\geq H$ by the definition. In the last case, suppose that the set of closed geodesics and free boundary geodesic arcs that $\Gamma$ is homotopic to is $\Gamma'$. We denote the regions bounded by $\Gamma'$ and $\Gamma$ by $\Omega'$ and $\Omega$, respectively, and denote by $d$ the distance between $\Gamma'$ and $\Gamma$.  Then \begin{equation}
			h_C=\frac{\ell(\Gamma)}{|\Omega|}=\frac{\ell(\Gamma')\cosh d}{|\Omega'|+\ell(\Gamma')\sinh d}\geq \frac{H}{H+1},
		\end{equation}
		since $H\leq \ell(\Gamma')/|\Omega'|$. This completes the proof.
	\end{proof}

	\section{on random surfaces}
	We will prove two probability results for the Cheeger's constant and the modified Jammes's constant in Section 5.1 and Section 5.2, respectively. Combining these two results we give a proof of the main Theorem \ref{maintheorem2}  in section 5.3.

	\subsection{Jammes' s constant on random hyperbolic surfaces}Assume $\{L_g\}_{g=2}^\infty$ is a sequence of $n$-tuples such that  $\lim\limits_{g\to \infty }\sum_{i=1}^n L_g^i =\infty$, $L_g^i>1$ and $\limsup\limits_{g\to \infty}\frac{\sum_{i=1}^nL_g^n}{\log g}<1$. We first show that as $g$ goes to infinity, a generic point $X\in \mathcal{M}_{g,n}(L_g)$, has uniformly positive Jammes's constant. More precisely,
	\begin{theorem}\label{randomHj}
		There exists some constant $c>0$ such that 
		$$
		\lim \limits_{g\to \infty}\Probb\Big(X_{g,n}\in \mathcal{M}_{g,n}(L_g); \tilde{h}_J(X_{g,n})>c \Big)=1.
		$$
	\end{theorem}
To prove this theorem, we need to prove  three propositions first.

	For any hyperbolic surface $X_{g,n}\in\mathcal{M}_{g,n}(L_g)$, if $d$ is the width  of  maximal half-collar of one boundary component $\eta$, then either 
there is a geodesic arc $\tau$ of length $d$ with two ends points lying on $\eta$ and another boundary component $\tilde{\eta}$ separately, or there is a geodesic arc $\gamma$ of length $2d$ with two ends lying on $\eta$.
For the first case, for $\epsilon>0$ small enough, the boundary  $\partial N_{\epsilon}(\eta\cup\tilde{\eta}\cup\tau)$ of the $\epsilon$-neighborhood $ N_{\epsilon}(\eta\cup\tilde{\eta}\cup\tau)$ of $\eta\cup\tilde{\eta}\cup \tau$ is homotopic to a geodesic $\xi$ that bounds a pair of pants along with $\eta$ and $\tilde{\eta}$. In particular, we have 
\begin{equation}\label{newcurves}\ell(\xi)\leq \ell(\eta)+\ell(\tilde{\eta})+2d,   \end{equation}
since $\xi$ is the unique closed geodesic in its homotopy class. 
 For the second case,
 the two ends on $\eta$ separate $\eta$ into $\eta_1,\eta_2$, so that loops $\eta_1\cup \gamma$ and $\eta_2\cup \gamma$ are homotopic to two closed geodesics $\alpha,\beta$, respectively. It is clear that $\{\eta,\alpha,\beta\}$ bounds a pair of pants. Similar to $(\ref{newcurves})$, we bound the length of $\alpha, \beta$ in the following proposition. 
	
	\begin{proposition}\label{collarbound}
		There exists a uniform constant $M>0$ such that if $\ell(\eta)>1$, then
		$$
		\ell(\alpha)+\ell(\beta)\leq 4d+\ell(\eta)+M.
		$$
	\end{proposition}
	\begin{proof}
		Cut the pair of pants bounded by $\{\eta,\alpha,\beta\}$ into two right-angled hexagons. Three disjoint boundary geodesic arcs are of length $\frac{\ell(\eta)}{2}$, $\frac{\ell(\alpha)}{2}$ and  $\frac{\ell(\beta)}{2}$. For convenience we use a geodesic arc to represent its length.  Denote the boundary geodesic arc connecting $\eta$ and $\alpha$ by $b$, then it follows \cite[Formula Glossary 2.4.1(i)]{buser2010geometry} that 
		$$
		\cosh b=\frac{  \cosh \frac{\beta}{2}+\cosh\frac{\alpha}{2}\cosh\frac{\eta}{2} }
		{\sinh \frac{\alpha}{2} \sinh \frac{\eta}{2}},
		$$
		then by \cite[Formula Glossary 2.3.4(i)]{buser2010geometry} we have $$
		\begin{aligned}
			\cosh d&=\sinh \frac{\alpha}{2}\sinh b\\
			=&\frac{\sqrt{(\cosh\frac{\beta}{2}+\cosh\frac{\alpha}{2}\cosh\frac{\eta}{2})^2-(\sinh \frac{\alpha}{2}\sinh\frac{\eta}{2})^2}}{\sinh \frac{\eta}{2}}\\
			=&\frac{\sqrt{\cosh^2\frac{\beta}{2}+\cosh^2\frac{\alpha}{2}+\cosh^2\frac{\eta}{2}+2\cosh\frac{\alpha}{2}\cosh\frac{\beta}{2}\cosh\frac{\eta}{2}-1}}{\sinh \frac{\eta}{2}},\\
		\end{aligned}
		$$
		so $d=\max\{\frac{\alpha-\eta}{2},\frac{\beta-\eta}{2},\frac{\alpha+\beta-\eta}{4},0\}+O(1)$ by the assumption that $\ell(\eta)>1$. 
	\end{proof}
	\begin{proposition}\label{lowerboundcollarwith}
		For  any fixed small $\epsilon>0$,
		let $(A)$ represent the condition that each boundary component of $X_{g,n}$ has a half-collar neighborhood of width $(\frac{1}{2}-\epsilon)\log g$, and all of them are disjoint. Then 
		$$
		\lim \limits_{g\to \infty}\Probb\Big(X_{g,n}\in \mathcal{M}_{g,n}(L_g);  (A) \text{ fails.} \Big)=0.
		$$
	\end{proposition}
	\begin{proof}
Set the $n$ boundary components of $X_{g,n}$ by $\eta_1,\cdots,\eta_n$.
		If $(A)$ fails, there are two cases. For the first case,
there exists  a geodesic $\xi$ bounding a pair of pants along with two boundary components $\eta_i$ and $\eta_j$ for  $i\neq j$ satisfying $\ell(\xi)\leq \ell(\eta_i)+\ell(\eta_j)+(1-2\epsilon )\log g$.
For the second case, by  Proposition \ref{collarbound}, there exist two geodesics $\alpha$ and $\beta$ that  bound a pair of pants along with some boundary component $\eta_k$ with $\ell(\alpha)+\ell(\beta)\leq \ell(\eta_k)+(2-4\epsilon)\log g+M$. Notice that in the second case, the complement of the pant bounded by $\alpha,\beta,\eta_k$ may be connected, or have two disconnected components.

We start with the first case.  It follows  from Mirzakhani's integral formula (MIF) (see \cite[Theorem 7.1]{Mirz07}), Lemma \ref{Mirz13lemma} and Lemma \ref{NWXlemma} that for $g$ sufficiently large,
\begin{equation}
\begin{aligned}
&\Probb\Big(X_{g,n}\in \mathcal{M}_{g,n}(L_g); \exists \ \xi \textit{ bounding a pair of pants with } \eta_1 \textit{ and }\eta_2 \\
&\textit{satisfying } \leq L_g^1+L_g^2+(1-2\epsilon)\log g \Big)\\
\prec& \frac{1}{V_{g,n}(L_g)}\int_{0}^{L_g^1+L_g^2+(1-2\epsilon)\log g}V_{0,3}(L_g^1,L_g^2,x)V_{g,n-1}(x,L_g^3,L_g^4,\cdots,L_g^n)xdx\\
\prec&\frac{1}{V_{g,n}}\cdot \prod_{i=1}^n\frac{L_g^i}{\sinh{\frac{1}{2}L_g^i}}\int_0^{L_g^1+L_g^2+(1-2\epsilon)\log g}V_{g,n-1}\cdot \prod_{i=3}^n\frac{\sinh{\frac{1}{2}L_g^i}}{L_g^i}\sinh\frac{x}{2}dx\\
\prec&\frac{1}{g}\frac{L_g^1L_g^2}{e^{\frac{1}{2}L_g^1+\frac{1}{2}L_g^2}}
e^{\frac{1}{2}[L_g^1+L_g^2+(1-2\epsilon )\log g]}\\=&\frac{L_g^1L_g^2}{g^{\frac{1}{2}+\epsilon}}=o(1),
\end{aligned}
\end{equation}
where  the implied constants only depend on $n$. For the second case that $\alpha\cup\beta$ bounds a pair of pants $P$ along with $\eta_i$, as mentioned above $X_{g,n}\setminus P$ is connected or has two connected components $X_{g_1,n_1+1}$ and $X_{g_2,n_2+1}$ with $g_1+g_2=g$, $n_1+n_2=n-1$, $\alpha\subset  X_{g_1,n_1+1}$ and $\beta\subset  X_{g_2,n_2+1}$.  We fix $n$ here so that the possible choices  of $n_1,n_2$ are finite, and the possible choices of  $n_1$ boundary components belonging to $X_{g_1,n_1+1}$ are finite.
 It follows  from Mirzakhani's integral formula (MIF) (see \cite[Theorem 7.1]{Mirz07}), Lemma \ref{Mirz13lemma}, Lemma \ref{NWXlemma} and Lemma \ref{mir13corollary3.7} that for $g$ sufficiently large,
		$$
		\begin{aligned}
			&\Probb\Big(X_{g,1}\in \mathcal{M}_{g,1}(L_g); \exists  \  \alpha,\beta  \textit{ bound a pair of pants along with } \\
& \gamma_1 \textit{ satisfying } \ell(\alpha)+\ell(\beta)\leq \ell(\eta_1)+(2-4\epsilon)\log g+M
 \Big)\\
			\prec&\frac{1}{V_{g,n}(L_g)}\int_{\mathbb{R}_+^2}1_{[0,L_g^1+(2-4\epsilon)\log g+M]}(x+y)V_{0,3}(L_g,x,y)\\
			\cdot&\Big[V_{g-1,n+1}(x,y,L_g^2,\cdots,L_g^n)+\sum_{g_1+g_2=g}\sum_{n_1+n_2=n-1}\sum_{\{i_1,\cdots,i_{n_1}\}\cup\{j_1,\cdots,j_{n_2}\}=\{2,\cdots,n\}}\\
&V_{g_1,n_1+1}(x,L_g^{i_1},\cdots,L_g^{i_{n_1}})V_{g_2,n_2+1}(y,L_g^{j_1},\cdots,L_g^{j_{n_2}})
\Big]xydxdy \\
			\prec&\frac{1}{V_{g,n}}\cdot\prod_{i=1}^n \frac{L_g^i}{\sinh{\frac{L_g^i}{2}}}\int_{x+y\leq L_g^1+(2-4\epsilon)\log g+M}\Big(V_{g-1,n+1}+\sum_{g_1+g_2=g}\sum_{n_1+n_2=n-1}\\
&\sum_{\{i_1,\cdots,i_{n_1}\}\cup\{j_1,\cdots,j_{n_2}\}=\{2,\cdots,n\}}
V_{g_1,n_1+1}V_{g_2,n_2+1}\Big)\prod_{i=2}^n\frac{\sinh{\frac{L_g^i}{2}}}{L_g^i}e^{\frac{x+y}{2}}dxdy \\
			\prec&\frac{1}{g}\cdot \frac{L_g^1}{e^\frac{L_g^1}{2}}\cdot (L_g^1+2\log g+M)e^{\frac{L_g^1}{2}+(1-2\epsilon)\log g}\\
			\prec&\frac{\log^2 g}{g^{2\epsilon}}=o(1),\\
		\end{aligned}
		$$
		where  the implied constants only depend on $n$. 
		
		The same estimates hold for the probability that there exists a simple closed geodesic $\xi$ bounding a pair of pants with any $\gamma_i,\gamma_j$ or a pair of simple closed geodesics $(\alpha,\beta)$ bounding a pair of pants with any $\gamma_k$ which  satisfies the length conditions, in the first calculation and the second calculation, respectively. This completes the proof.
	\end{proof}
Unlike surfaces with one boundary component (see the proof of Theorem \ref{maintheorem1}), large width of the half-collars of all boundary components of $X_{g,n}$  do not always imply large modified Jammes's constant as we have explained in the introduction. We need the following proposition in the spirit of Lemma 31 in \cite{NWX20} that will be important in proving a positive lower bound of $\tilde{h}_J(\Sigma).$

Define $N_{g_0,k}(X_{g,n},L)$ to be the number of  multicurves $\gamma=\cup_{i=1}^s\gamma_i$  that separate $X_{g,n}$ into exactly two subsurfaces with one part homotopic to  $S_{g_0,k}$  and the total length of $\gamma$ is no more than $L$. 

\begin{proposition}\label{XgnL1}
For $L>1$ and any integer $m>0$, there exists a constant $c=c(n,m)>0$ independent of $L$ and $g$ such that  \[
\sum_{m\leq|2g_0-2+k|\leq \frac{1}{2}(2g-2+n)}\Expp\Big[ N_{g_0,k}(X_{g,n},L) \Big]\leq c(n,m)\frac{e^{2L}}{g^m}.
\]
In particular, if $L=(\frac{1}{2}-\epsilon)\log g$, the left-hand side of the above inequality goes to zero as $g$ tends to infinity.
\end{proposition}
\begin{proof}
For the multicurve $\gamma=\cup_{i=1}^s\gamma_i$ that cuts off a $S_{g_0,k}$, it separates $X_{g,n}$ into two components, i.e., $S_{g_0,k}\cup S_{g+1-g_0-s,n-k+2s}$.
The boundary $\partial S_{g_0,k}$ contains $k-s$ boundary components of $X_{g,n}$. We compute the number of such $\gamma$'s that bound a $S_{g_0,k}$ along with boundaries $\eta_1,\cdots,\eta_{k-s}$.
 Applying 
Mirzakhani's integral formula (MIF) (see \cite[Theorem 7.1]{Mirz07}) to $1_{[0,L]}(x_1+\cdots+x_s)$ we have$$
\begin{aligned}
&\Expp\Big[ \#\big\{\gamma=\cup_{i=1}^s \gamma_i: \gamma \textit{ bound a }  S_{g_0,k} \textit{ along with } \eta_1,\cdots,\eta_{k-s},\  \ell(\gamma)\leq L  \big\} \Big]\\
&\prec\frac{1}{V_{g,n}(L_g)s!}\int_{\sum_{i=1}^sx_i\leq L}V_{g_0,k}(L_g^1,\cdots,L_g^{k-s},x_1,\cdots,x_s)\\
&\cdot V_{g+1-g_0-s,n-k+2s}   (L_g^{k-s+1},\cdots, L_g^n,x_1,\cdots,x_s)   x_1\cdots x_s dx_1\cdots dx_s,\\
\end{aligned}$$
where the implied constants are uniform.
Since $\sum_{i=1}^nL_g^i\leq \log g$ for large $g$, by   Lemma \ref{NWXlemma} for large $g$  we have \[\begin{aligned}
&\frac{V_{g_0,k}(L_g^1,\cdots,L_g^{k-s},x_1,\cdots,x_s) V_{g+1-g_0-s,n-k+2s}   (L_g^{k-s+1},\cdots, L_g^n,x_1,\cdots,x_s)  }{V_{g,n}(L_g)}\\
\prec&\frac{V_{g_0,k}V_{g+1-g_0-s,n-k+2s}}{V_{g,n}}\cdot\prod_{i=1}^s\frac{2\sinh\frac{x_i}{2}}{x_i},
\end{aligned}
\]
where  the implied constants only depend on $n$.
 Since $n$ is fixed, the possibility of the $k-s$ components are finite.
By  Lemma \ref{Mirz13lemma} we have \begin{equation}\label{eq-xgn-Ng0k}
\begin{aligned}
&\Expp\Big[N_{g_0,k}(X_{g,n},L)\Big]\\
\prec&\sum_{s\leq k}\frac{V_{g_0,k} V_{g+1-g_0-s,n-k+2s}  }{V_{g,n}}\frac{1}{s!}\int_{\sum_{i=1}^sx_i\leq L}e^{\sum_{i=1}^sx_i}dx_1\cdots dx_s\\
\prec&\sum_{s\leq k}\frac{V_{g_0,k}V_{g+1-g_0,n-k}}{V_{g,n}}\frac{e^{L}L^s}{(s!)^2}\\
\prec &\frac{V_{g_0,k}V_{g+1-g_0,n-k}}{V_{g,n}}e^{2L},
\end{aligned}
\end{equation}
where  the implied  constants only depend on $n$.
Notice that we can use the fact that $\sinh t\leq t\cosh t$ for any $t\geq 0$ in the first inequality.
 By Lemma \ref{Mirz13lemma} we have $$
V_{g_0,k}\leq V_{g_0+\frac{k-k^\prime}{2},k^\prime} 
$$ with $k-k^\prime $ even and $K^\prime\in\{1,2,3\}$. Similarly, $$
V_{g+1-g_0,n-k}\leq V_{g+1-g_0+\frac{n-k-k^{\prime\prime}}{2},k^{\prime\prime}}
$$ with $n-k-k^{\prime\prime}$ even and $k^{\prime\prime}\in\{1,2,3\}$. Then 
 it follows Lemma \ref{mir13corollary3.7} we have $$
\begin{aligned}
&\sum_{m\leq |2g_0-2+k|\leq \frac{1}{2}(2g-2+n)}\Expp\Big[ N_{g_0,k}(X_{g,n},L) \Big]\\
\prec&\sum_{\begin{matrix}
m\leq |2g_1-2+k_1|\leq \frac{1}{2}(2g-2+n)\\
2g_1+k_1+2g_2+k_2=2g+2+n\\
k_1,k_2\in\{1,2,3\}
\end{matrix}}\frac{(1+g_1)V_{g_1,k_1}V_{g_2,k_2}}{V_{g,n}}e^{2L}\\
\prec&\frac{1}{g^m}e^{2L},
\end{aligned}
$$
where the implied constants only depend on $n$ and $m$. This completes the proof of the proposition.
\end{proof}
\begin{remark}
As a corollary of Proposition \ref{XgnL1}, for a subset $\mathcal{M}\subset \mathcal{M}_{g,n}(L_g)$ containing surfaces $X_{g,n}$ where the minimal lengths of multicurve separating $X_{g,n}$ are  shorter than $(\frac{1}{2}-\epsilon)\log g$, the Weil-Petersson probability $\lim_{g\to\infty}\Probb\Big(\mathcal{M}\Big)=0$.  While this length can be replaced by $(2-\epsilon)\log g$ by doing more dedicated estimates similar to Lemma 30 in \cite{NWX20},  Proposition \ref{XgnL1} is enough for our purpose.
\end{remark}

We now prove Theorem \ref{randomHj}.
	\begin{proof}[Proof of Theorem \ref{randomHj}] Assume that $\partial X_{g,n}=\{\eta_1,\cdots,\eta_n\}$.
		Note that we assume  $\sum_{i=1}^nL_g^i\leq \log g$ for large $g$. The first part in the proof of Theorem \ref{maintheorem1} says that if $(A)$ holds,  then for  subset $\Omega$ of $X_{g,n}$ with $|\Omega|\leq \frac{1}{2}|X_{g,n}|$, $\partial_E\Omega\neq \emptyset$ and all components of $\Omega^c$ intersect with $\partial X_{g,n}$, when $\partial_I\Omega$ is  only contained in the disjoint union of half-collar $N(\eta_i,(\frac{1}{2}-\epsilon)\log g)$, or  passes through at least one half-collar, we have $$\frac{\ell(\partial_I\Omega)}{\ell(\partial_E\Omega)}\geq \frac{1}{2}-\epsilon.$$  
If these two cases do not happen and $\ell(\partial_I\Omega)< (\frac{1}{2}-\epsilon)\log g$, we shrink all components of $\partial \Omega$ to the unique geodesics in their homotopy classes. Then we will find a multicurve $\gamma=\cup_{i=1}^k\gamma_i$ that   separates $X_{g,n}$ into two parts
 with total length less than $(\frac{1}{2}-\epsilon)\log g$. 
Take $L=(\frac{1}{2}-\epsilon)\log g$ in Proposition \ref{XgnL1}, since all $N_{g_0,k}(X_{g,n},L)$ are integer-valued random variables,
we have \[\begin{aligned}
&\lim_{g\to\infty}\Probb\Big[  X_{g,n}\in\mathcal{M}_{g,n}(L_g),\exists\textit{ multicurve } \gamma \textit{ separates } X_{g,n}, \ell(\gamma)\leq L \Big]\\
\leq& \lim_{g\to\infty} \sum_{1\leq |2g_0-2+k|\leq \frac{1}{2}(2g-2+n)}\Expp\Big[ N_{g_0,k}(X_{g,n},L) \Big] \\
=&0.\\
\end{aligned}
\]
 Hence by Proposition \ref{lowerboundcollarwith} and the above inequality the theorem holds by choosing $c$ to be $\frac{1}{2}-\epsilon$ for any small $\epsilon>0$.
	\end{proof}

\vskip.3cm
	\subsection{The Cheeger's constant on random hyperbolic surfaces}
The condition on $L_g=(L_g^1,\cdots,L_g^n)$ that $\limsup\limits_{g\to\infty}\frac{\sum_{i=1}^nL_g^i}{\log g}<1$ is essential in this subsection.
	We will show that as $g$ goes to infinity, a generic surface $X_{g,n}\in \mathcal{M}_{g,n}(L_g)$ has uniformly positive Cheeger's constant. In fact, we have the following theorem.
	\begin{theorem}\label{randomH}
		For any constant $c_1<\frac{\ln 2}{2\pi}$, $$
		\lim\limits_{g\to\infty}\Probb\Big(X_{g,n}\in \mathcal{M}_{g,n}(L_g);  H(X_{g,n})<c_1
		\Big)=0.
		$$
	\end{theorem}
	We will prove this theorem at the end of this section. Firstly, combining this with Lemma \ref{hgeqHH+1}, we directly have the following corollary which will be used in the proof of the main theorem of this article.
	\begin{corollary}\label{randomhc}
		For any constant $c<\frac{\ln 2}{2\pi+\ln 2}$, $$
		\lim\limits_{g\to\infty}\Probb\Big(X_{g,n}\in \mathcal{M}_{g,n}(L_g);  h_C(X_{g,n})<c 
		\Big)=0.
		$$
	\end{corollary}
We start with some preliminaries before proving Theorem \ref{randomH}.
	Fix $0<c_1<1$.
	If $H(X_{g,n})<c_1$, then there exists a half surface $X\subset X_{g,n}$ such that $\ell(\partial_I X)\leq c_1|X|$ and $|X|\leq \frac{1}{2}|X_{g,n}|$ with $X$ and $X_{g,n}-X$ connected. By doubling $X_{g,n}$ along $\partial X_{g,n}$,  $\partial_I X$ become closed geodesics, so they bound an embedded subsurface. Therefore $|X|$ is an integer multiple of $\pi$.
	Assume that $X_{g,n}$ satisfies  the condition $(A)$ and $|X|=k\pi$ with $k\leq 2g-2+n$, 
	then if  $\partial_I X$ contains $m$ geodesic arcs meeting orthogonally with $\partial X_{g,n}$, we have
	\begin{equation} \label{boundm}
		m\cdot(1-2\epsilon)\log g\leq c_1|X|=c_1k\pi. 
	\end{equation}
	Consider the $\delta$-neighborhood $N_\delta(X\cup \partial X_{g,n})$ of $X\cup \partial X_{g,n}$ for $\delta$ small enough. Then each boundary loop of $N_\delta(X\cup\partial X_{g,n})$ is homotopic to a simple closed geodesic. This is similar to the method of dealing with nonsimple closed geodesics in \cite{WX21,NWX20,mirzakhani2019lengths}. We shall prove that there do not exist two boundary loops that are homotopic to each other and no boundary loop is homotopically trivial for large $g$. If such loops exist, the geodesic arcs corresponding to them will bound a topological cylinder or a disk outside $X$ with a total length less than $\sum_{i=1}^nL_g^i+kc_1\pi$. Since $X_{g,n}-X$ is connected, then the cylinder or disk is just $X_{g,n}-X$. By isoperimetric inequalities (See \eg section $8.1$ in \cite{buser2010geometry}), $|X_{g,n}-X|\leq\sum_{i=1}^n L_g^i+kc_1\pi.$ However by our assumption $|X_{g,n}-X|\geq \frac{1}{2}|X_{g,n}|$, which is a contradiction for large $g$. Therefore $\partial X_{g,n}\cup X$ deforms to an embedded subsurface $\tilde{X}$ of $X_{g,n}$ containing $X$. By (\ref{boundm}) for any fixed $\epsilon^\prime>0$ and  for $g$ large enough, $$k\pi\leq |\tilde{X}|\leq k\pi+m\pi\leq (1+\epsilon^\prime)k\pi\leq \frac{1+\epsilon^\prime}{2}|X_{g,n}|,$$
	and $$\ell(\partial_I \tilde{X})\leq \sum_{i=1}^nL_g^n+c_1|X|\leq  \sum_{i=1}^nL_g^n+c_1|\tilde{X}|.$$ 
	Therefore  \begin{equation}\label{hcprobineq}
		\begin{aligned}
			&\Probb\Big(X_{g,n}\in \mathcal{M}_{g,n}(L_g); H(X_{g,n})<c_1 \Big)\\
			\leq &\Probb\Big(X_{g,n}\in \mathcal{M}_{g,n}(L_g); (A) \ \ \text{fails} \Big)\\
			+&\Probb\Big(X_{g,n}\in \mathcal{M}_{g,n}(L_g); \exists\ \tilde{X}, |\chi(\tilde{X})|\leq\frac{1+\epsilon^\prime}{2}|\chi(X_{g,n})|\\
			&\text{and } \ell(\partial_I \tilde{X})\leq  \sum_{i=1}^nL_g^i+c_1|\tilde{X}|
			\Big).
		\end{aligned}
	\end{equation}
	We have proved that the first term converges to $0$ as $g\to \infty$. In what follows we prove the second term also goes to 0 as $g\rightarrow 0$. We define 
	\begin{equation*}
		\tilde{W}_k=\begin{cases}
			V_{\frac{k}{2},2}\ \ \ \ \ \text{if $k$ is even}\\
			V_{\frac{k+1}{2},1}\ \ \text{ if $k$ is odd}.
		\end{cases}
	\end{equation*}
	It  is slightly different from $W_k$ in \cite{NWX20}.
	
	\begin{proposition}\label{boundchik}
		For any constant $c_2$ larger than $c_1$, we have $$
		\begin{aligned}
			&\Probb\Big(X_{g,n}\in \mathcal{M}_{g,n}(L_g);\ \exists\ \tilde{X},|\chi(\tilde{X})|=m,\\
			&\ell(\partial_I \tilde{X})\leq \sum_{i=1}^n L_g^i+2c_1m\pi \Big)
			\prec\frac{e^{2\pi mc_2+\frac{c_2}{c_1}\sum_{i=1}^nL_g^i}\tilde{W}_m\tilde{W}_{2g-2+n-m}}{V_{g,n}}.\\
		\end{aligned}
		$$
	\end{proposition}
	\begin{proof}
		By the same calculation as \eqref{eq-xgn-Ng0k}, if $2g_0-2+k=m$, take $L=\sum_{i=1}^nL_g^i+2c_1m\pi$, then $$
		\begin{aligned}
			&\Probb\Big(X_{g,n}\in \mathcal{M}_{g,n}(L_g);\ \exists\ \tilde{X} \textit{ of type } S_{g_0,k},\ 
			\ell(\partial_I \tilde{X})\leq L\Big)\\
\prec&\Expp\Big( N_{g_0,k}(X_{g,n},L)      \Big)\\
\prec&\sum_{s\leq k}\frac{V_{g_0,k}V_{g+1-g_0,n-k}}{V_{g,n}}\frac{e^LL^s}{(s!)^2}
\prec\frac{V_{g_0,k}V_{g+1-g_0,n-k}}{V_{g,n}}e^{L+2L^\frac{1}{2}}.\\
		\end{aligned}
		$$
This follows by  the fact that $\sum_{s}\frac{(x^2)^s}{(s!)^2}\leq (\sum_{s}\frac{x^s}{s!})^2\leq e^{2x}$ for $x=L^{\frac{1}{2}}$. The constants implied here only depend on $n$. Summing it over all possible $(g_0,k)$, we have $$
		\begin{aligned}
			&\Probb\Big(X_{g,n}\in \mathcal{M}_{g,n}(L_g);\ \exists\ \tilde{X},|\chi(\tilde{X})|=m,\\
			&\ell(\partial_I \tilde{X})\leq \sum_{i=1}^n L_g^i+2c_1m\pi \Big)
			\prec\frac{ m e^{L+2L^{\frac{1}{2}}}\tilde{W}_m\tilde{W}_{2g-2+n-m}}{V_{g,n}}\\
\prec&\frac{e^{2\pi mc_2+\frac{c_2}{c_1}\sum_{i=1}^nL_g^i}\tilde{W}_m\tilde{W}_{2g-2+n-m}}{V_{g,n}},
		\end{aligned}
		$$
where the implied constants depend only on $n,c_1,c_2$.
	\end{proof}
	
	\begin{proposition}\label{boundepsilonprime}
		The following estimate holds $$
		\sum_{m\leq\frac{1+\epsilon^\prime}{2}(2g-2+n)}e^{2\pi mc_2}\tilde{W}_m\tilde{W}_{2g-2+n-m}\prec\sum_{m\leq \frac{1}{2}(2g-2+n)} e^{2\pi mc_2\frac{1+\epsilon^\prime}{1-\epsilon^\prime}}\tilde{W}_m\tilde{W}_{2g-2+n-m}.
		$$
	\end{proposition}
	\begin{proof}
		If $\frac{1}{2}(2g-2+n)\leq m\leq\frac{1+\epsilon^\prime}{2}(2g-2+n)$, then take $m^\prime =2g-2+n-m$ and we have $\frac{m}{m^\prime}\leq \frac{1+\epsilon^\prime}{1-\epsilon^\prime}.$ Thus
		$$
		e^{2\pi mc_2}\tilde{W}_m\tilde{W}_{2g-2+n-m}\leq e^{2\pi m^\prime c_2\frac{1+\epsilon^\prime}{1-\epsilon^\prime}}\tilde{W}_{m^\prime}\tilde{W}_{2g-2+n-m^\prime}.
		$$
This completes the proof.
	\end{proof}

We are now in position to give a proof of Theorem \ref{randomH}.
	\begin{proof}[Proof of Theorem\ref{randomH}]
		By  Lemma \ref{Mirz13lemma} and Lemma \ref{mir13corollary3.7}, for both $k=1,2$,  if we assume that
		$4\pi\frac{1+\epsilon^\prime}{1-\epsilon^\prime}c_2<2\ln2$, then 
		$$
			\sum_{m\leq \frac{1}{2}(2g-2+n)} e^{2\pi mc_2\frac{1+\epsilon^\prime}{1-\epsilon^\prime}}\tilde{W}_m\tilde{W}_{2g-2+n-m}
\prec\frac{V_{g,n}}{g}.
		$$
		Combining it with  Proposition \ref{boundchik} and Proposition \ref{boundepsilonprime} yields that, if we assume that
		$4\pi\frac{1+\epsilon^\prime}{1-\epsilon^\prime}c_2<2\ln2$, then
		$$ 
		\begin{aligned}
			&\Probb\Big(X_{g,n}\in \mathcal{M}_{g,n}(L_g);\ \exists\ \tilde{X},\ |\chi(\tilde{X})|\leq\frac{1+\epsilon^\prime}{2}|\chi(X_{g,n})| \\
			&\ell(\partial_I \tilde{X})\leq \sum_{i=1}^n L_g^i+c_1|\tilde{X}|
			\Big)
			\prec\frac{e^{\frac{c_2}{c_1}\|L_g\|_1}}{g}.
		\end{aligned}
		$$
		Since $c_2>c_1$ can be arbitrarily close to $c_1$ and $\epsilon^\prime>0$ can be arbitrarily small, and $\limsup\limits_{g\to\infty}\frac{\|L_g\|_1}{\log g}<1$, now theorem follows from (\ref{hcprobineq}).
	\end{proof}
	
	\vskip.3cm
	\subsection{Proof of Theorem \ref{maintheorem2}}
	By Theorem \ref{randomHj} and Corollary \ref{randomhc}, for some fixed positive numbers $c,c_1$ we have $$\lim \limits_{g\to \infty}\Probb\Big(X_{g,n}\in \mathcal{M}_{g,n}(L_g);\ \tilde{h}_J(X_{g,n})>c, \ h_C(X_{g,n})>c_1 \Big)=1.$$
	It follows from Lemma \ref{Jammes} that $$\lim \limits_{g\to \infty}\Probb\Big(X_{g,n}\in \mathcal{M}_{g,n}(L_g); \
	\sigma_1(X_{g,n})\geq\frac{c\cdot c_1}{4}
	\Big)=1.$$ This completes the proof of Theorem \ref{maintheorem2} by taking $C=\frac{c\cdot c_1}{4}$. 
		
	\bibliographystyle{plain}
	\bibliography{Reference} 
	
\end{document}